\documentclass[11pt]{article}  

\usepackage{setspace}
\usepackage{float}
\usepackage{amsmath,amsfonts,amssymb}
\usepackage{amsthm}
\usepackage{graphicx}
\usepackage{epstopdf}
\usepackage{caption}
\usepackage{url}
\usepackage[margin=2.5cm]{geometry}
\usepackage{hyperref}

\theoremstyle{plain}
\newtheorem{theorem}{Theorem}
\newtheorem{lemma}{Lemma}

\theoremstyle{definition}
\newtheorem{definition}{Definition}

\newtheorem{remark}{Remark}

\newcommand\blfootnote[1]{%
\begingroup
\renewcommand\thefootnote{}\footnote{#1}%
\addtocounter{footnote}{-1}%
\endgroup
}

\begin{document}	
	
\title{Pharmacokinetic/Pharmacodynamic Anesthesia Model\\ 
Incorporating psi-Caputo Fractional Derivatives\blfootnote{This 
is a preprint version of the paper published open access in 
\emph{Computers in Biology and Medicine} 
(Print ISSN: 0010-4825; Online ISSN: 1879-0534),
doi: \href{https://doi.org/10.1016/j.compbiomed.2023.107679}{https://doi.org/10.1016/j.compbiomed.2023.107679}.}}
	
\author{Mohamed Abdelaziz Zaitri$^{1, 2}$\\
\texttt{zaitri@ua.pt}
\and
Hanaa Zitane$^1$\\
\texttt{h.zitane@ua.pt}
\and Delfim F. M. Torres$^{1,}$\thanks{Corresponding author.}\\
\texttt{delfim@ua.pt}}
\date{$^1$Center for Research and Development in Mathematics and Applications (CIDMA),
Department of Mathematics, University of Aveiro, 3810-193 Aveiro, Portugal\\[0.3cm]
$^2$Department of Mathematics, University of Djelfa, 17000 Djelfa, Algeria}
	
\maketitle


\begin{abstract}
We present a novel Pharmacokinetic/Pharmacodynamic (PK/PD) model for the induction 
phase of anesthesia, incorporating the $\psi$-Caputo fractional derivative. By employing 
the Picard iterative process, we derive a solution for a nonhomogeneous $\psi$-Caputo 
fractional system to characterize the dynamical behavior of the drugs distribution 
within a patient's body during the anesthesia process. 
To explore the dynamics of the fractional anesthesia model, we 
perform numerical analysis on solutions involving various functions of $\psi$ and 
fractional orders. All numerical simulations are conducted using the MATLAB 
computing environment. Our results suggest that the $\psi$ functions 
and the fractional order of differentiation have an important role
in the modeling of individual-specific characteristics,
taking into account the complex interplay between drug concentration 
and its effect on the human body. This innovative model 
serves to advance the understanding of personalized drug responses 
during anesthesia, paving the way for more precise and tailored 
approaches to anesthetic drug administration.

\medskip
	

\medskip

\end{abstract}

	
\section{Introduction}

Pharmacokinetic/Pharmacodynamic (PK/PD) modeling is a mathematical approach 
used in pharmacology to study the relationship between drug concentrations (Pharmacokinetics) 
and their effects on the body (Pharmacodynamics) \cite{Morse,Singh}. 
The PK/PD models help researchers and clinicians to understand how drugs 
are absorbed, distributed, metabolized, and eliminated from the body \cite{Beck}.

The PK/PD models integrate Pharmacokinetic and Pharmacodynamic data to characterize 
the time course of drug action \cite{Meibohm} . These models can be simple or complex, 
depending on the drug's characteristics and the purpose of the modeling. The parameters 
of these models were fitted by Schnider et al. in \cite{Schnider}. Some common 
types of PK/PD models include:
\begin{enumerate}
\item Linear PK/PD models. The basic structure of a linear PK/PD model consists 
of two main components: the Pharmacokinetic component, which describes the drug 
concentration-time profile in the body, and the Pharmacodynamic component, 
which relates to the drug concentration to the observed effect \cite{Gabrielsson}.

\item Non-linear PK/PD models. These models can incorporate various components, 
such as saturable drug elimination, receptor binding kinetics, and indirect response models. 
These models provide a more accurate representation of the concentration-effect 
relationship and can be used to optimize dosing regimens and predict drug responses 
in different populations \cite{Jusko}.
	
\item Mechanistic PK/PD models. These models incorporate known biological mechanisms, 
such as drug-receptor interactions, enzyme kinetics, or signal transduction pathways. 
They provide a more detailed representation of drug action but require more data 
and knowledge about the underlying biology \cite{Nielsen}.
	
\item Population  Pharmacokinetic  models. These models are used to describe  
the  time  course  of  drug  exposure  in  patients  and  to  investigate  
sources  of  variability  in  patient  exposure.  They  can  be  used  to  
simulate  alternative  dose  regimens,  allowing  for  informed  assessment 
of dose regimens before study conduct \cite{delAmo,Mould}. 
\end{enumerate}
 
In recent years, the field of fractional derivatives has emerged as a promising approach 
to model and understand complex biological processes characterized by non-integer 
order dynamics \cite{Ionescu,LiF,Pandey}. 
This unique mathematical framework has found diverse applications in various areas of biology, 
where traditional integer-order calculus falls short in capturing the intricacies 
of these systems \cite{Baleanu3,MR4560070}.

One prominent field where fractional derivatives have made significant contributions is Neurobiology. 
By employing fractional calculus, researchers have been able to delve into the dynamics of neural systems 
with a greater level of realism. This includes modeling the behavior of neurons, synaptic transmission, 
and the propagation of nerve impulses. The incorporation of fractional derivatives enables the consideration 
of memory effects and non-local behavior, providing a more accurate representation of neural processes \cite{Magin}.

Another area where fractional derivatives have shown their efficacy is in Biomedical Signal Processing. 
Biomedical signals, such as electroencephalograms (EEG), electrocardiograms (ECG), and blood pressure signals, 
are often complex and exhibit non-integer order dynamics. By utilizing fractional order filters and operators, 
meaningful information can be extracted from these signals, leading to improved analysis and interpretation. 
Fractional calculus has proven to be a valuable tool in enhancing our understanding 
of these vital physiological signals~\cite{Ferdi}.

Furthermore, the field of Cancer Modeling has witnessed the application 
of fractional derivatives \cite{Pachauri}. 
Tumor growth and the intricate interactions between cancer cells and the immune system present complex 
dynamics that can be effectively captured using fractional order models. By incorporating non-local 
effects and memory into the modeling process, fractional derivatives provide a comprehensive framework 
for studying cancer progression. This approach has the potential to shed light on the underlying 
mechanisms and aid the development of novel therapeutic strategies~\cite{Arfan,Turkyilmazoglu}.

The field of Pharmacokinetics and Pharmacodynamics plays a crucial role in understanding the behavior 
of drugs in biological systems \cite{Singh}. Traditional PK/PD models have predominantly relied on integer-order 
derivatives to describe various processes involved in drug absorption, distribution, metabolism, 
and elimination \cite{Eleveld,Morse,ZST}. However, these models often fall short in accurately capturing 
the complexity of Pharmacokinetic behaviors \cite{Brunton,Hennion}.

Fractional calculus offers a promising alternative by providing a more precise representation 
of these intricate dynamics \cite{Sopasakis}. A number of studies have shown that certain 
drugs follow an anomalous kinetics that can hardly be represented by classical models. 
Indeed, fractional-order pharmacokinetics models have proved to be better suited to represent 
the time course of these drugs in the body and also they can describe memory effects and 
a power-law terminal phase. Therefore, they give rise to more complex kinetics that better 
reflects the complexity of the human body. In \cite{Hennion}, a fractional one-compartment 
model with a continuous intravenous infusion is considered, where it allows to determine 
how the infusion rate influences the total amount of drug in the compartment. Moreover, 
in the case of multiple dosing administration, recurrence relations for the doses and 
the dosing times that also prevent drug accumulation are presented. Hence, in \cite{copot}, 
a PK model was introduced employing a fractional-order approach akin to mammillary dynamics. 
This model was specifically designed to incorporate considerations of tissue entrapment, 
thereby altering the anticipated drug concentration profiles. The proposed model shows 
a limitation in data fit profiles, without transgressing the fundamental principles 
of mass balance and physiological states. The mathematical study of the amount 
of drug administered as a continuous intravenous infusion or oral dose for fractional-order 
mammillary type models is investigated in \cite{MTCG}.

For pharmacokinetic and pharmacodynamic models, the first fractional one 
was introduced in \cite{Verotta}. By utilizing the Caputo fractional derivative, 
the authors presented a fresh perspective on the intricate relationships 
between drug dosages, absorption rates, and therapeutic outcomes. Since then, 
some few applications of fractional PK/PD models have appeared in the literature 
\cite{Atici,IonescuC}. In \cite{Atici}, the authors study the discrete and discrete 
fractional representation of a PK/PD model describing tumor growth and anti-cancer 
effects in continuous time considering a specific times scale, while in \cite{IonescuC} 
a fractional PK/PD model in anesthesia is developed to describe the nonlinear
characteristics of the PK/PD patient models.

In this article, we propose a novel fractional four-compartmental PK/PD model 
for the induction phase of anesthesia, employing the $\psi$-Caputo fractional 
derivative. This innovative approach involves replacing each ordinary derivative 
in the classical PK/PD model investigated in \cite{Said,ZST} with the 
$\psi$-Caputo fractional derivative of order $\alpha \in (0,1]$ \cite{Almeida}. 

By incorporating fractional derivatives into the model, we aim to capture the intricate 
dynamics of anesthesia induction more accurately. Furthermore, based on the Picard 
iterative process, we establish a solution for 
the nonhomogeneous $\psi$-Caputo fractional differential 
system of equations in terms of matrix Mittag-Leffler functions. The solution provides a deeper 
understanding of the dynamics and behavior of the anesthesia process, taking into account 
the complex interplay between drug concentration and its effect on the body.

To validate the effectiveness of our proposed model, we conduct numerical simulations. 
Specifically, we determine the optimal anesthesia dosage for a 53-year-old male weighing 77~Kg 
and measuring 177~cm, while considering the minimum treatment time \cite{Said}. 
These simulations enable us to evaluate the performance and applicability 
of the fractional PK/PD model in a practical context.

Through the article, we aim to demonstrate the potential of fractional derivatives 
in advancing our understanding of biology, particularly in the field of PK/PD modeling. 
By incorporating fractional calculus into these models, we can unlock new insights 
and improve our ability to predict and optimize drug behavior within biological systems.

The paper is organized as follows. In Section~\ref{Preliminaries}, we provide a review 
of several definitions and properties of fractional calculus that are essential for our 
subsequent discussions (Section~\ref{subsec:2.1}); and we recall the 
Pharmacokinetic and Pharmacodynamic models proposed by Bailey et al. \cite{Bailey} 
and Schnider \cite{Schnider} (Section~\ref{Section:PK:PD:model}), 
discussing their bispectral index (Section~\ref{subsec:2.3}) 
and the equilibrium points (Section~\ref{subsec:2.4}).
Our original contributions are then given in Section~\ref{sec:MR}:
we obtain a solution to a general linear nonhomogeneous $\psi$-Caputo fractional system
(Section~\ref{Solution:nonhomegenous:fractional:systems});
we introduce a novel PK/PD model for the induction phase 
of anesthesia based on the $\psi$-Caputo fractional derivative 
(Section~\ref{fractional:PK/PD:model}); and finally we compute 
the model parameters using the Schnider model \cite{Schnider}, presenting the numerical results 
of the fractional PK/PD model corresponding to different $\psi$ functions and fractional orders
(Section~\ref{Numerical:example}). We conclude with Section~\ref{sec:disc}, 
discussing the implications and limitations of our results, and with Section~\ref{sec:conc}, 
summarizing our findings and outlining potential directions for future research.

	
\section{Preliminaries}
\label{Preliminaries}

In this section, we recall several definitions, 
properties of fractional calculus,
and the classical PK/PD model,
that will be used in the sequel.

\subsection{Fundamental definitions and results}
\label{subsec:2.1}

Throughout the paper, $\psi$ designates a function of class $C^{1}[a, b]$ 
such that $\psi'(t) > 0$, for all $t \in [a, b]$.	

\begin{definition}[See \cite{der2}]
\label{DefInt} 
The left $\psi$-Riemann-Liouville
fractional integral of a function $f$ of order $\alpha\in (0,1)$ is defined by
$$I_{a}^{\alpha,\psi}f(t)=\dfrac{1}{\Gamma(\alpha)}
\int_{a}^{t}\psi'(s)\left(\psi(t)-\psi(s)\right)^{\alpha-1}f(s)\, \mathrm{d}s,
$$	
where $\Gamma(\cdot)$ is the Euler Gamma function.
\end{definition}

\begin{remark}
We remark that $\Gamma(x + 1) = x\Gamma(x)$, for all $x > 0$, 
and for any positive integer $n$ we have
$\Gamma(n + 1) = n!$.	
\end{remark}

\begin{definition}[See \cite{der2}]
The $\psi$-Caputo fractional derivative of a function $f$ 
of order $\alpha\in (0,1)$ can be defined as follows:
$$	
{}^C \!D_{a}^{\alpha,\psi}f(t)=\dfrac{1}{\Gamma(1-\alpha)}
\int_{a}^{t}\left(\psi(t)-\psi(a)\right)^{-\alpha}f'(s)\, \mathrm{d}s.
$$
\end{definition}

We have the following properties of the fractional operators 
with respect to function $\psi$.

\begin{lemma}[See \cite{der2}]
\label{lemm1}
Let	$\Re(\alpha)>0$ and $\Re(\beta)>0$. Then,
$$
I_{a}^{\alpha,\psi}\left(f(x)-f(a)   \right)^{\beta-1}(t)
=\dfrac{\Gamma(\beta)}{\Gamma(\beta+\alpha)}
\left(f(t)-f(a)\right)^{\beta-\alpha-1}.
$$	
\end{lemma}

\begin{theorem}[See \cite{Almeida}]
\label{Composition} 
Let $\alpha \in (0,1)$ and $f\in C^{1}(a,b)$. Then,
$$
I_{a}^{\alpha,\psi}{}^C \!D_{a}^{\alpha,\psi}f(t)=f(t)-f(a).
$$
\end{theorem}

The Mittag--Leffler function appears naturally in the solution of fractional differential 
equations and in its various applications: see \cite{Mittaglefller,MR3914412} and references therein.

\begin{definition}[See \cite{Mittaglefller}]
The Mittag--Leffler function of one parameter, of a matrix $A$, is defined as
\begin{equation}
\label{eq1}
E_{\alpha}(A)=\sum_{l=0}^{+\infty}\frac{A^{l}}{\Gamma(\alpha l+1)},
\quad Re(\alpha)>0.
\end{equation}
\end{definition}

\begin{definition}[See \cite{Mittaglefller}]
The Mittag--Leffler function of two parameters, of a matrix $A$, is defined as
\begin{equation}
\label{eq2}
E_{\alpha,\alpha'}(A)=\sum_{l=0}^{+\infty}\frac{A^{l}}{\Gamma(\alpha l+ \alpha')},
\quad Re(\alpha)>0,~\alpha'>0.
\end{equation}	
\end{definition}

\begin{remark}
The matrix exponential function is a special case 
of the matrix Mittag--Leffler function \cite{MR3914412}.
For $\alpha'=1$, we have $E_{\alpha,1}(A)=E_{\alpha}(A)$ 
and $E_{1,1}(A)=e^{A}$.		
\end{remark}

We now recall the notion of generalized convolution integral.

\begin{definition}[See  \cite{JaradAbdel}]
\label{Defa}
Let $f$ and $g$ be two functions which are piecewise continuous at
any interval $[a, b]$ and of exponential order. The generalized convolution
of $f$ and $g$ is defined by	
$$
\left(f\ast_{\psi}g \right)(t)= \int_{a}^{t}f(s)g\left(\psi^{-1}(\psi(t)
+\psi(a)-\psi(s))\right)\psi'(s)\, \mathrm{d}s.	   
$$	
\end{definition}


\subsection{The classical PK/PD model: state of the art}
\label{Section:PK:PD:model}

The Pharmacokinetic/Pharmacodynamic (PK/PD) model comprises four compartments: 
intramuscular blood $(y_1(t))$, muscle $(y_2(t))$, fat $(y_3(t))$, and the effect site $(y_4(t))$. 
The inclusion of the effect site compartment (representing the brain) is necessary to account for 
the finite equilibration time between the central compartment and the central nervous system 
concentrations \cite{Bailey}. This model is employed to characterize the distribution of drugs 
within a patient's body and can be mathematically described by a four-dimensional dynamical system:
\begin{equation} 
\label{model:PK/PD} 
\left\{
\begin{array}{l l}
\dot{y}_1(t)= -(a_{1\,0}+a_{1\,2}+a_{1\,3})\,y_1(t)+a_{2\,1}\,y_2(t)+ a_{3\,1}\,y_3(t)+u(t),\\		
\dot{y}_2(t)= a_{1\,2}\,y_1(t) -a_{2\,1} \, y_2(t),\\	
\dot{y}_3(t)= a_{1\,3}\,	y_1(t)-a_{3\,1}\,y_3(t),\\		
\dot{y}_4(t)= \frac{a_{e\,0}}{v_1}\,y_1(t) -a_{e\,0}\, y_4(t),
\end{array}\right.
\end{equation}
subject to the initial conditions
\begin{equation}
\label{initial:state}
y_1(0)=y_2(0)=y_3(0)=y_4(0)= 0,
\end{equation}
where $y_1(t)$, $y_2(t)$, $y_3(t)$ and $y_4(t)$ represent, respectively, 
the masses of the propofol in the compartments of blood, muscle, fat, 
and effect site at time $t$. The control $u(t)$ represents the continuous 
infusion rate of the anesthetic. The parameters $a_{1\,0}$ and $a_{e\,0}$ represent, 
respectively, the rate of clearance from the central compartment and the effect site. 
The parameters $a_{1\,2}$, $a_{1\,3}$, $a_{2\,1}$, $a_{3\,1}$ and $a_{e\,0}/v_1$ are 
the transfer rates of the drug between compartments. All these parameters depend on age, 
weight, height and gender, and the relations can be found in Table~\ref{Table:Schnider}.

A schematic diagram of the dynamical control system \eqref{model:PK/PD} 
is given in Figure~\ref{schema01}. 
\begin{figure}[!htb]
\begin{center}
\includegraphics[scale=0.8]{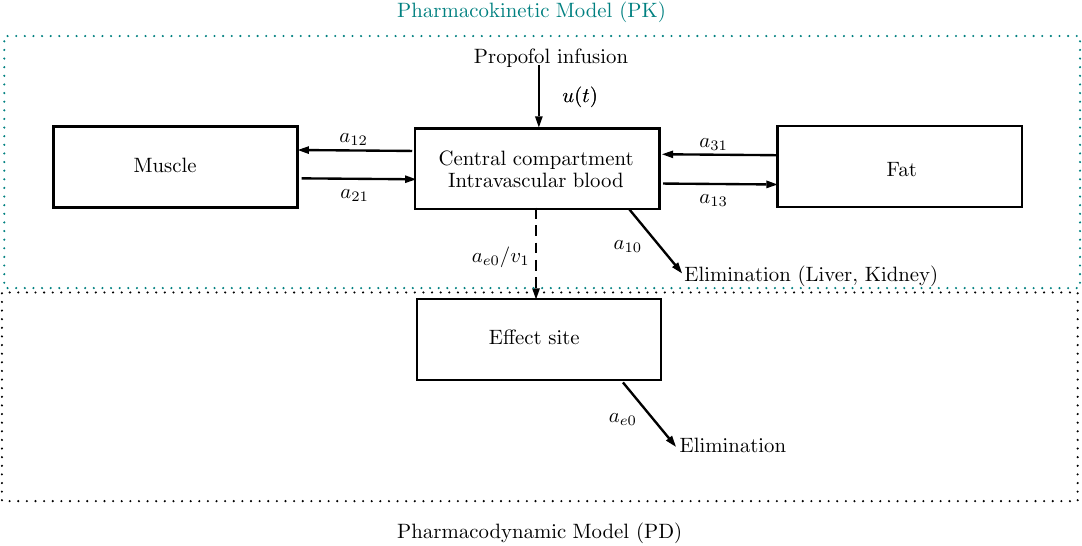}
\caption{Schematic diagram of the PK/PD model with the effect site compartment 
of Bailey and Haddad \cite{Bailey}.}
\label{schema01}
\end{center}
\end{figure}

\begin{table}[!htb]
\caption{Parameter values for model \eqref{model:PK/PD}
according with Schnider models \cite{Schnider}.}
\label{Table:Schnider}
\centering
\begin{tabular}{|c|c|} \hline
{\bf Parameter} & {\bf Estimation} \\ \hline
$a_{10}\,(min^{-1})$ 
& $0.443+0.0107\, (\text{weight}-77)-0.0159\, 
(\text{LBM}-59)+0.0062\, (\text{height}-177)$\\ \hline
$a_{12}\,(min^{-1})$ 
& $0.302-0.0056 \, (\text{age}-53)$\\ \hline
$a_{13}\,(min^{-1})$ 
& 0.196 \\ \hline
$a_{21}\,(min^{-1})$ 
& $\left( 1.29-0.024\, (\text{age}-53)\right) 
/\left( 18.9-0.391\, (\text{age}-53)\right) $ \\ \hline
$a_{31}\,(min^{-1})$ & 0.0035 \\ \hline
$a_{e0}\,(min^{-1})$ & 0.456 \\ \hline
$v_1\,(l)$ & 4.27 \\ \hline
\end{tabular}
\end{table}

Following Schnider et al. \cite{Schnider}, 
the lean body mass (LBM) is calculated using the James
formula, which performs satisfactorily in normal and
moderately obese patients, but not so well for
severely obese cases \cite{James}. 
The James formula calculates LBM as follows:
\begin{eqnarray}
\text{for Male},\, 
\text{LBM}&=&1.1\times \text{weight}
-128\times\left( \dfrac{\text{weight}}{\text{height}}\right)^2,\\	
\text{for Female},\, \text{LBM}&=&1.07\times \text{weight}
-148\times\left( \dfrac{\text{weight}}{\text{height}}\right)^2.
\end{eqnarray}


\subsection{The bispectral index (BIS)}
\label{subsec:2.3}

The bispectral index (BIS) serves as an indicator of anesthesia depth, obtained by analyzing 
the electroencephalogram (EEG) signal and reflecting the effect site concentration of $y_4(t)$. 
It provides a quantitative measure of a patient's level of consciousness, ranging from 0 
(indicating no cerebral activity) to 100 (representing a fully awake patient). According to 
clinical guidelines, maintaining BIS values within the range of 40 to 60 is considered essential 
to ensure adequate general anesthesia during surgical procedures \cite{Avidan}. BIS values below 
40 indicate a deep hypnotic state, while values above 60 may increase the risk of awareness 
under anesthesia. Thus, it is crucial to closely monitor and regulate the BIS value within the optimal 
range of 40 to 60 to prevent unintended consciousness during anesthesia 
and ensure patient safety \cite{Evers}.

Empirically, the BIS can be described by a decreasing sigmoid function, 
as outlined by Bailey et al. \cite{Bailey}:
\begin{equation}
\label{BIS}
BIS(y_4(t))
=BIS_0\left(1-\dfrac{y_4(t)^\gamma}{y_4(t)^\gamma+EC_{50}^\gamma} \right), 
\end{equation}
where the parameters in the BIS model have specific interpretations. 
Function $BIS_0$ gives the BIS value of an awake patient that is typically set to 100; 
$EC_{50}$ corresponds to the drug concentration at which $50\%$ of the maximum effect 
is achieved, while $\gamma$ is a parameter that captures the degree of nonlinearity 
in the model. According to Haddad et al. \cite{Haddad}, typical values for these parameters 
are $EC_{50} = 3.4$ $mg/l$, and $\gamma = 3$.


\subsection{The equilibrium point}
\label{subsec:2.4}

The equilibrium points are computed by putting the right-hand side of the four equations
given in \eqref{model:PK/PD} equal to zero with the condition
\begin{equation}
y_4=EC_{50}.
\end{equation}
It results that the equilibrium point 
$y_e=\left(y_{e\,1},y_{e \, 2},y_{e\,  3},y_{e\,  4} \right)$ is given by
\begin{equation}
y_{e\,  1  }=v_1\,EC_{50}, \quad
y_{e\,  2}= \frac{a_{1\,2}\,v_1\,EC_{50}}{a_{2\,1}}, \quad
y_{e\,  3}=\frac{a_{1\,3}\,v_1\,EC_{50}}{a_{3\,1}}, \quad
y_{e\,  4}=EC_{50},	
\end{equation}
and the value of the continuous infusion rate for this equilibrium is
\begin{equation}
u_{e}=a_{1\,0}\,v_1 \,EC_{50}. 
\end{equation}

The fast state is defined by
\begin{equation}
\label{fast:final:state}
y_{eF}(t)=(y_1(t),y_4(t)).
\end{equation}

For more information on the classical PK/PD model we refer the interested
reader to \cite{Zabi}.


\section{Main Results}
\label{sec:MR}

We begin by using the Picard iterative process to prove a series 
solution to a linear nonhomogeneous $\psi$-Caputo fractional system:
see Theorem~\ref{thm:MR}, in Section~\ref{Solution:nonhomegenous:fractional:systems}. 
Then, we generalize the state-of-the-art PK/PD model \eqref{model:PK/PD} 
by introducing in Section~\ref{fractional:PK/PD:model} a more general $\psi$-Caputo 
fractional PK/PD model that is covered by our Theorem~\ref{thm:MR}.
We finish our new results in Section~\ref{Numerical:example}, by investigating 
numerically the new fractional model and comparing the efficacy of function $\psi$.


\subsection{Solution of linear nonhomogeneous $\psi$-Caputo fractional systems}
\label{Solution:nonhomegenous:fractional:systems}

Consider the following linear nonhomogeneous fractional equation:
\begin{equation}
\label{Maineq}
{}^C \!D_{a}^{\alpha,\psi}y(t)=Ay(t)+u(t), \quad t>a,
\end{equation}
subject to the initial condition
\begin{equation}
\label{Condintial}
y(a)=y_{0},
\end{equation}  
where ${}^C \!D_{a}^{\alpha,\psi}$ is the $\psi$-Caputo fractional 
derivative of order $\alpha \in (0,1]$, such that 
$$
{}^C \!D_{a}^{\alpha,\psi}y(t)=\left[{}^C \!D_{a}^{\alpha,\psi}y_{1}(t),
{}^C \!D_{a}^{\alpha,\psi}y_{2}(t),\ldots,{}^C \!D_{a}^{\alpha,\psi}y_{n}(t)\right]^{T},
$$ 
$A$ is a $n\times n$  matrix, $u(t)=\left[u_{1}(t),u_{2}(t),\ldots,u_{n}(t)\right]^{T}$ 
is a piecewise continuous integrable function on $[a,+\infty)$, and the initial condition is
$y(a)=\left[y_{1}(a),y_{2}(a),\ldots,y_{n}(a)\right]^{T}$.

\begin{lemma}
\label{lemma3}
Let $p\in \mathbb{N}$, $\alpha \in(0,1]$, and $f$ be a piecewise continuous 
function of exponential order at
any interval $[a, b]$. Then, 
$$
I_{a}^{(p+1)\alpha,\psi}f(t)=\dfrac{\left(\psi(t)
-\psi(a)\right)^{p\alpha+\alpha-1}}{\Gamma(p\alpha+\alpha)}\ast_{\psi}f(t).
$$
\end{lemma}

\begin{proof}
Follows by using the change of variable $z=\psi^{-1}\left(\psi(t)+\psi(a)-\psi(s)\right)$, 
Definition~\ref{Defa}, and performing direct calculations.
\end{proof}

\begin{lemma}
\label{lemma2} 
Let $\alpha \in(0,1]$ and $C$ be a constant. Then, one has	
$$
I_{a}^{\alpha,\psi}C=\dfrac{C}{\Gamma(\alpha+1)}\left(\psi(t)-\psi(a)\right)^{\alpha}.  
$$
\end{lemma}

\begin{proof}
From Definition~\ref{DefInt}, we have 
\begin{equation*}
\begin{split}
I_{a}^{\alpha,\psi}C
&=\dfrac{C}{\Gamma(\alpha)}
\int_{a}^{t}\psi'(s)\left(\psi(t)-\psi(s)\right)^{\alpha-1}\, \mathrm{d}s\\
&=\dfrac{C}{\Gamma(\alpha)} [\alpha^{-1}\left(\psi(t)-\psi(s)\right)]^{t}_{a}\\
&=\dfrac{C}{\Gamma(\alpha+1)}\left(\psi(t)-\psi(a)\right)^{\alpha},
\end{split}	
\end{equation*}	
and the proof is complete.
\end{proof}

Now, we shall utilize the Picard iterative process \cite{Duan} to formulate a series 
solution to \eqref{Maineq}--\eqref{Condintial}.

\begin{theorem}
\label{thm:MR}	
The solution of the initial value problem~\eqref{Maineq}--\eqref{Condintial} 
can be given in series form as
\begin{equation}
\label{Sol1}
y(t)=\sum_{l=0}^{\infty}\frac{A^{l}\left(\psi(t)-\psi(a)\right)^{l\alpha}}{\Gamma(l\alpha+1)}y(a)
+\sum_{l=0}^{\infty}\frac{A^{l}\left(\psi(t)-\psi(a)\right)^{l\alpha
+\alpha-1}}{\Gamma(l\alpha+\alpha)}\ast_{\psi}u(t).
\end{equation}
\end{theorem}

\begin{proof}
Applying the fractional integration operator $I_{a}^{\alpha,\psi}$ to both sides of 
equation~\eqref{Maineq}, and using Theorem~\ref{Composition}, we obtain the following expression:
$$
y(t)=y(a)+AI_{a}^{\alpha,\psi}y(t)+I_{a}^{\alpha,\psi}u(t).
$$
Let $\phi_{k}$ be the $k$th approximate solution 
with the initial one given by
$$
\phi_{0}(a)=y(a)
$$ 
and, for $k\geq1$, the recurrent formula 
\begin{equation}
\label{Recurrent}
\phi_{k}(t)=y(a)+AI_{a}^{\alpha,\psi}\phi_{k-1}(t)+I_{a}^{\alpha,\psi}u(t)
\end{equation}
being satisfied. From formula~\eqref{Recurrent} and Lemma~\ref{lemma2}, one has 
\begin{equation*}
\begin{split}
\phi_{1}(t)&=y(a)+\dfrac{A\left(\psi(t)-\psi(a)\right)^{\alpha}}{\Gamma(\alpha+1)}y(a)
+I_{a}^{\alpha,\psi}u(t),\\
\phi_{2}(t)&=y(a)+\dfrac{A\left(\psi(t)-\psi(a)\right)^{\alpha}}{\Gamma(\alpha+1)}y(a)
+\dfrac{A^{2}\left(\psi(t)-\psi(a)\right)^{2\alpha}}{\Gamma(2\alpha+1)}y(a)
+AI_{a}^{2(\alpha,\psi)}u(t)+I_{a}^{\alpha,\psi}u(t),\\ \vdots\\
\phi_{k}(t)&= \sum_{l=0}^{k}\frac{A^{l}\left(\psi(t)-\psi(a)\right)^{l\alpha}}{\Gamma(l\alpha +1)}y(a)
+\sum_{l=0}^{k-1}A^{l}I_{a}^{(l+1)(\alpha,\psi)}u(t).
\end{split}	
\end{equation*} 
By virtue of Lemma~\ref{lemma3} and by taking the limit $k\longrightarrow \infty$ for $\phi_{k}(\cdot)$, 
we obtain the series formula \eqref{Sol1} for the solution 
of~\eqref{Maineq}--\eqref{Condintial}.
\end{proof}

Note that, in terms of the matrix Mittag--Leffler functions \eqref{eq1} and \eqref{eq2}, 
the solution~\eqref{Sol1} may be written as
\begin{equation}
\label{Sol2}
y(t)=E_{\alpha}\left(A(\psi(t)-\psi(a))^{\alpha}
\right)y(a)+(\psi(t)-\psi(a))^{\alpha-1}
E_{\alpha,\alpha}\left(A(\psi(t)-\psi(a))^{\alpha}\right)\ast_{\psi}u(t).
\end{equation}


\subsection{A fractional PK/PD model}
\label{fractional:PK/PD:model}

Motivated by system \eqref{model:PK/PD}, we introduce here our $\psi$-Caputo fractional 
Pharmacokinetic/Pharmacodynamic model, which is obtained by replacing each ordinary 
derivative in the system~\eqref{model:PK/PD} 
by the $\psi$-Caputo fractional derivative of order $\alpha \in (0,1]$. 
Then, our proposed PK/PD model can be expressed by the following 
four-dimensional fractional dynamical system: 
\begin{equation} 
\label{model2} 
\left\{
\begin{array}{l l}
{}^C \!D_{0}^{\alpha,\psi}y_1(t)= -(a_{1\,0}+a_{1\,2}+a_{1\,3})\,y_1(t)
+a_{2\,1}\,y_2(t)+ a_{3\,1}\,y_3(t)+u_{1}(t),\\		
{}^C \!D_{0}^{\alpha,\psi}y_2(t)= a_{1\,2}\,y_1(t) -a_{2\,1} \, y_2(t),\\	
{}^C \!D_{0}^{\alpha,\psi}y_3(t)= a_{1\,3}\,	y_1(t)-a_{3\,1}\,y_3(t),\\		
{}^C \!D_{0}^{\alpha,\psi}y_4(t)= \frac{a_{e\,0}}{v_1}\,y_1(t) -a_{e\,0}\, y_4(t),
\end{array}\right.
\end{equation}
subject to the initial conditions
\begin{equation}
\label{initial2}
y_1(0)=y_2(0)=y_3(0)=y_4(0)= 0.
\end{equation}
According to the dynamical system \eqref{Maineq}, one may write system 
\eqref{model2}--\eqref{initial2} in a matrix form as follows:
\begin{equation}
\label{Model2Matrix}
{}^C \!D_{0}^{\alpha,\psi}y(t)=A\,y(t)+B\,u_1(t)
\end{equation}
with $y(t)=\left[y_1(t),y_2(t),y_3(t),y_4(t)\right]^{T}
\in \mathbb{R}^4$, $y(0)=\left[0,0,0,0\right]^{T}$, 
$$
A=\left(
\begin{array}{cccc}
-(a_{1\,0}+a_{1\,2}+a_{1\,3}) & a_{2\,1} & a_{3\,1} & 0 \\
a_{1\,2} & -a_{2\,1} & 0 & 0 \\
a_{1\,3} & 0 & -a_{3\,1} & 0 \\
\frac{a_{e\,0}}{v_1} & 0 & 0 & -a_{e\,0} 
\end{array}
\right)
\quad \text{ and } \quad
B=\left(
\begin{array}{c}
1 \\
0  \\
0  \\
0 
\end{array}
\right).
$$	
One mentions that the continuous infusion rate $u_{1}(t)$ is to be chosen 
in such a way to transfer the system \eqref{model2} from the initial state 
(wake state) to the fast final state (anesthetized state).

\begin{remark}
If $\psi(t)=t$ and $\alpha=1$, then the fractional system \eqref{model2} 
reduces to the classical PK/PD model \eqref{model:PK/PD} \cite{ZST}. 
\end{remark}


\subsection{Numerical simulations}
\label{Numerical:example}

To administer anesthesia to a 53-year-old man weighing 77~Kg and measuring 177~cm, 
we utilize our proposed fractional PK/PD system described by:
\begin{equation}
\label{Example1:Time:Optimal}
\begin{cases}
{}^C \!D_{0}^{\alpha,\psi}y(t)=A\,y(t)+B\,u_1(t),\\
 y(0)=(0,\,0,\,0,\,0)^{T},
\end{cases}
\end{equation}
where, according with Table~\ref{Table:Schnider} and \cite{Said}, 
the matrix $A$ is taken as
\begin{equation}
A=\left(
\begin{array}{cccc}
-0.9175  &  0.0683 &   0.0035&         0\\
0.3020 &  -0.0683  &       0    &     0\\
0.1960  &       0 &  -0.0035     &    0\\
0.1068   &      0&         0  & -0.4560
\end{array}
\right)
\quad \text{ and }\quad 
B=\left(
\begin{array}{c}
1 \\
0  \\
0  \\
0 
\end{array}
\right),
\end{equation}	
with 
\begin{equation}
\label{Optimal:control:u:3}
u_{1}(t)=\begin{cases}
106.0907\, mg/min& \hbox{if}\,\, 0\leq t<0.5467, \\
0 & \hbox{if}\,\, 0.5467<t\leq 1.8397 .
\end{cases}	
\end{equation}

From Theorem~\ref{thm:MR} of Section~\ref{Solution:nonhomegenous:fractional:systems},
written in form \eqref{Sol2}, the solution of system \eqref{Example1:Time:Optimal} is given by 
\begin{equation}
\label{Sol3}
y(t)=E_{\alpha}\left(A(\psi(t)-\psi(0))^{\alpha}
\right)y(0)+(\psi(t)-\psi(0))^{\alpha-1}E_{\alpha,\alpha}\left(A(\psi(t)-\psi(0))^{\alpha}\right)
\ast_{\psi}u(t)
\end{equation}
with $u(t)= B u_1(t) = \left[u_{1}(t),0,0,0\right]^{T}$. 

Figure~\ref{Solution} showcases the solutions derived from the fractional PK/PD model 
\eqref{Example1:Time:Optimal}, considering the function $\psi(t)=t$ and exploring different 
fractional order values: $\alpha=1$, $\alpha=0.95$, $\alpha=0.9$, and $\alpha=0.85$. 
In Figure~\ref{BISI}, the curves represent the controlled BIS (Bispectral Index) associated 
with the optimal continuous infusion rate of the administered anesthetic $u(t)$. 
It is noteworthy that when the function $\psi(t)=t$ and the fractional order is set to $\alpha=1$, 
then the obtained results resemblance those derived from the classical PK/PD model \eqref{model:PK/PD}. 
However, altering the fractional orders introduces variations in the degree of anesthesia. 
The recorded values for all fractional orders fell within the range of 40 to 50 (corresponding 
to the classical model), thus ensuring the condition of anesthesia. Nevertheless, it is crucial 
to acknowledge that lower fractional order values entail a higher risk of awareness during anesthesia.
\begin{figure}[!htb]
\begin{center}
\includegraphics[width=\textwidth,height=\textheight,keepaspectratio]{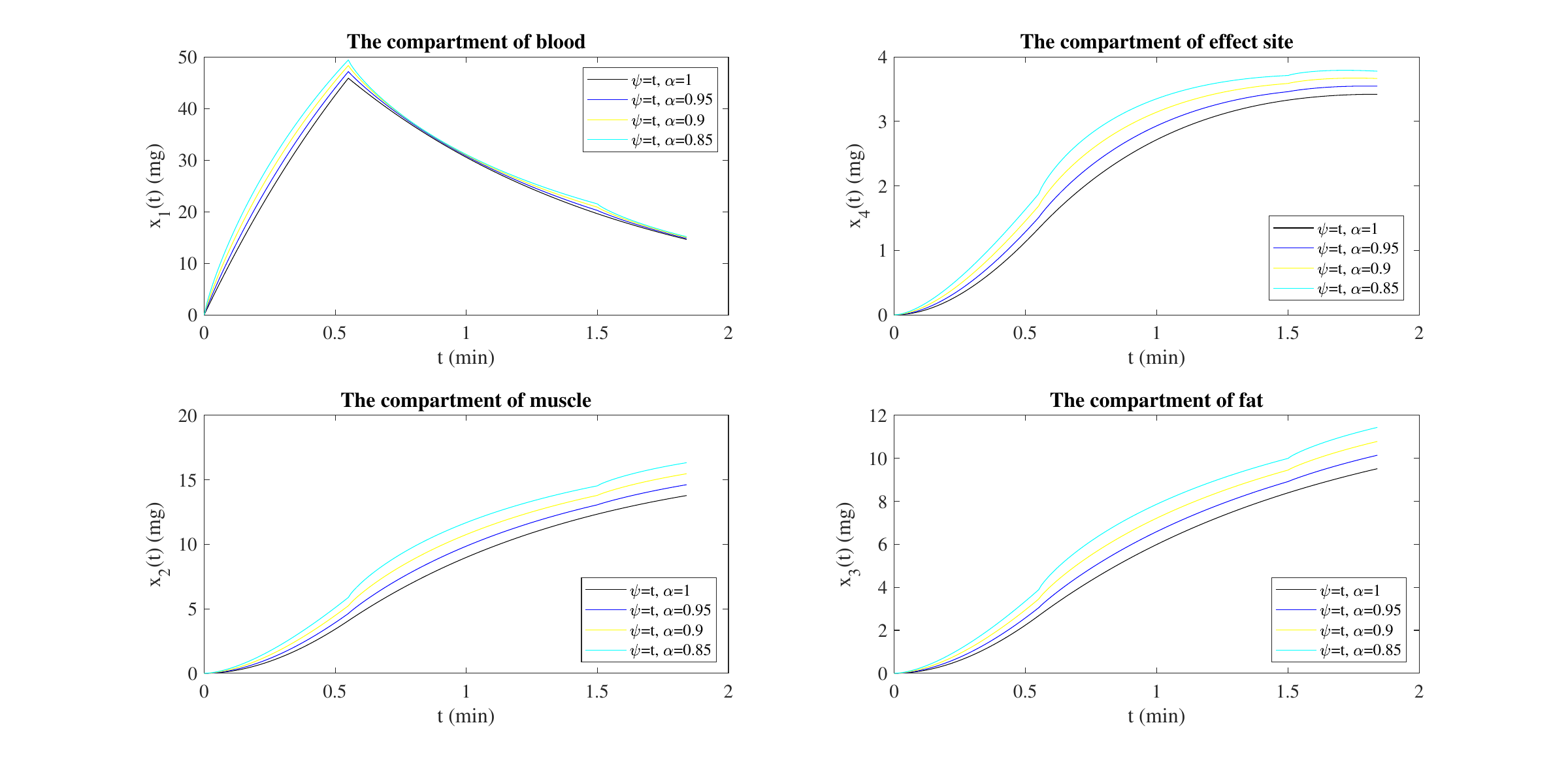}
\caption{Analysis of the fractional PK/PD model \eqref{Example1:Time:Optimal} 
with functions $\psi(t)=t$ for fractional orders $\alpha=1$, 
$\alpha=0.95$, $\alpha=0.9$  and $\alpha=0.85$.}
\label{Solution}
\end{center}
\end{figure}
\begin{figure}[!htb]
\begin{center}
\includegraphics[width=\textwidth,height=\textheight,keepaspectratio]{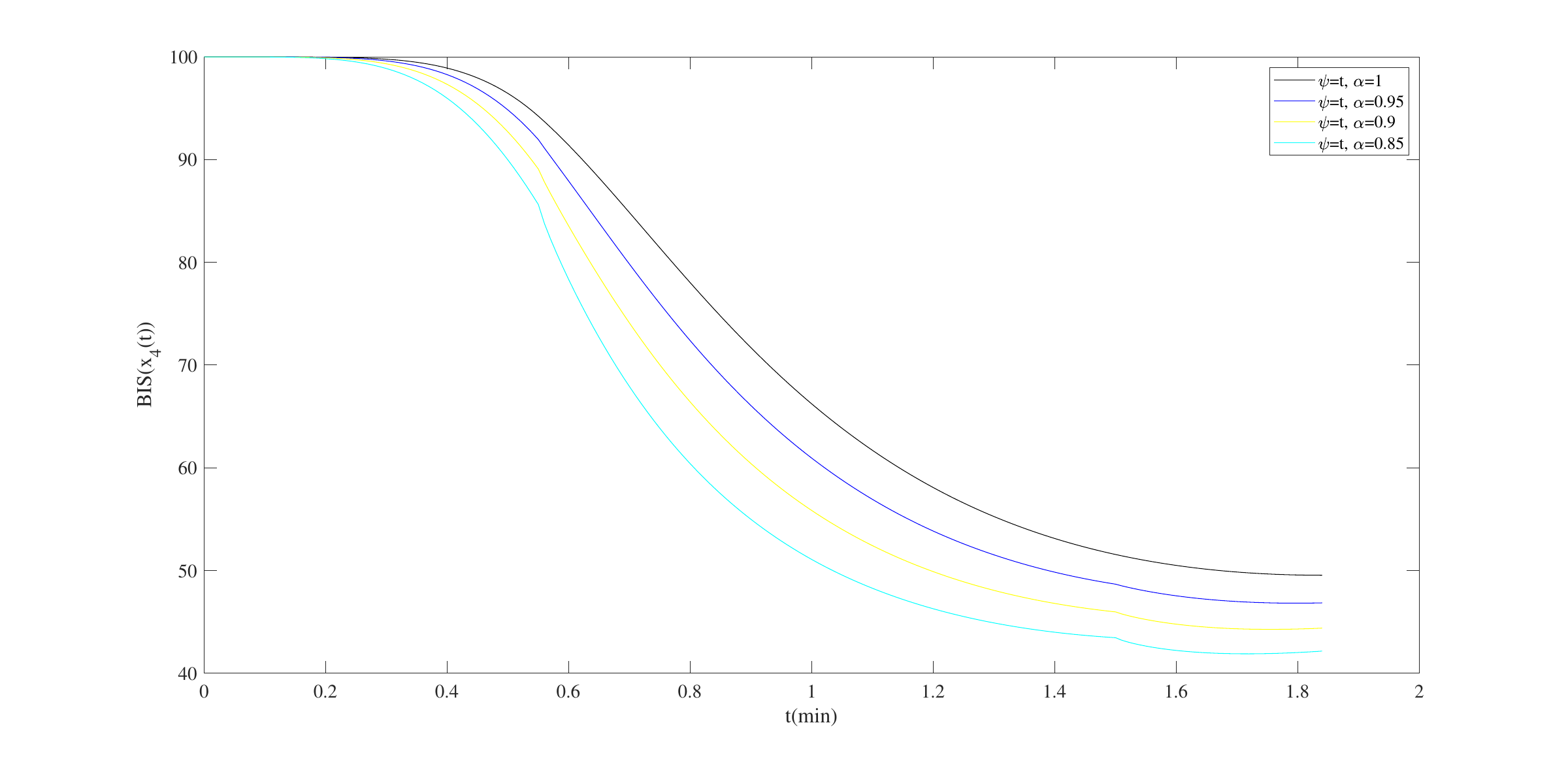}
\caption{Analysis of controlled BIS with functions $\psi(t)=t$ for fractional orders 
$\alpha=1$, $\alpha=0.95$, $\alpha=0.9$ and $\alpha=0.85$.} 
\label{BISI}
\end{center}
\end{figure}

Figure~\ref{Solution:alpha:1} illustrates the solutions of the fractional PK/PD model 
\eqref{Example1:Time:Optimal} associated with the functions $\psi(t)=t$, $\psi(t)=\sqrt{t}$, 
$\psi(t)=t^2$, and $\psi(t)=t+0.2$, when considering a fractional order of $\alpha=1$. 
The graphs shown in Figure~\ref{BISI:alpha:1} depict the controlled BIS corresponding 
to a specific value of the fractional order $\alpha=1$,  under functions  $\psi(t)=t$, 
$\psi(t)=\sqrt{t}$, $\psi(t)=t^2$ and $\psi(t)=t+0.2$.
It is observed that selecting functions $\psi(t)=\sqrt{t}$ and $\psi(t)=t^2$ does not yield 
satisfactory anesthesia results. On the other hand, employing the functions $\psi(t)=t$ 
and $\psi(t)=t+0.2$ leads to favorable anesthesia outcomes. In subsequent simulations, 
we will maintain the functions $\psi(t)=t$ and $\psi(t)=t+0.2$ while altering the fractional orders. 

\begin{figure}[!htb]
\begin{center}
\includegraphics[width=\textwidth,height=\textheight,keepaspectratio]{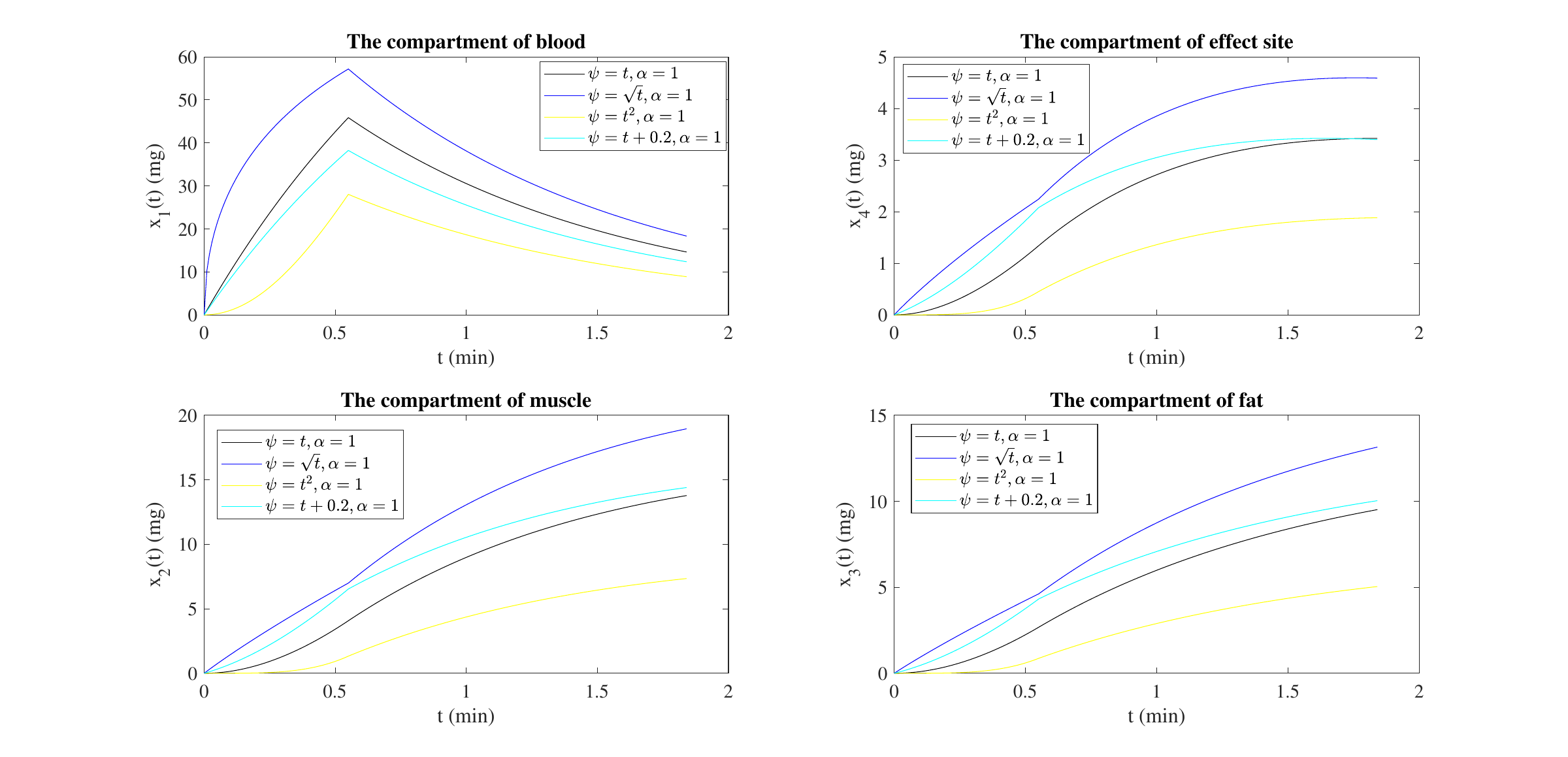}
\caption{Analysis of the fractional PK/PD model \eqref{Example1:Time:Optimal} with functions 
$\psi(t)=t$, $\psi(t)=\sqrt{t}$, $\psi(t)=t^2$ and $\psi(t)=t+0.2$ 
for fractional order $\alpha=1$.} 
\label{Solution:alpha:1}
\end{center}
\end{figure}
\begin{figure}[!htb]
\begin{center}
\includegraphics[width=\textwidth,height=\textheight,keepaspectratio]{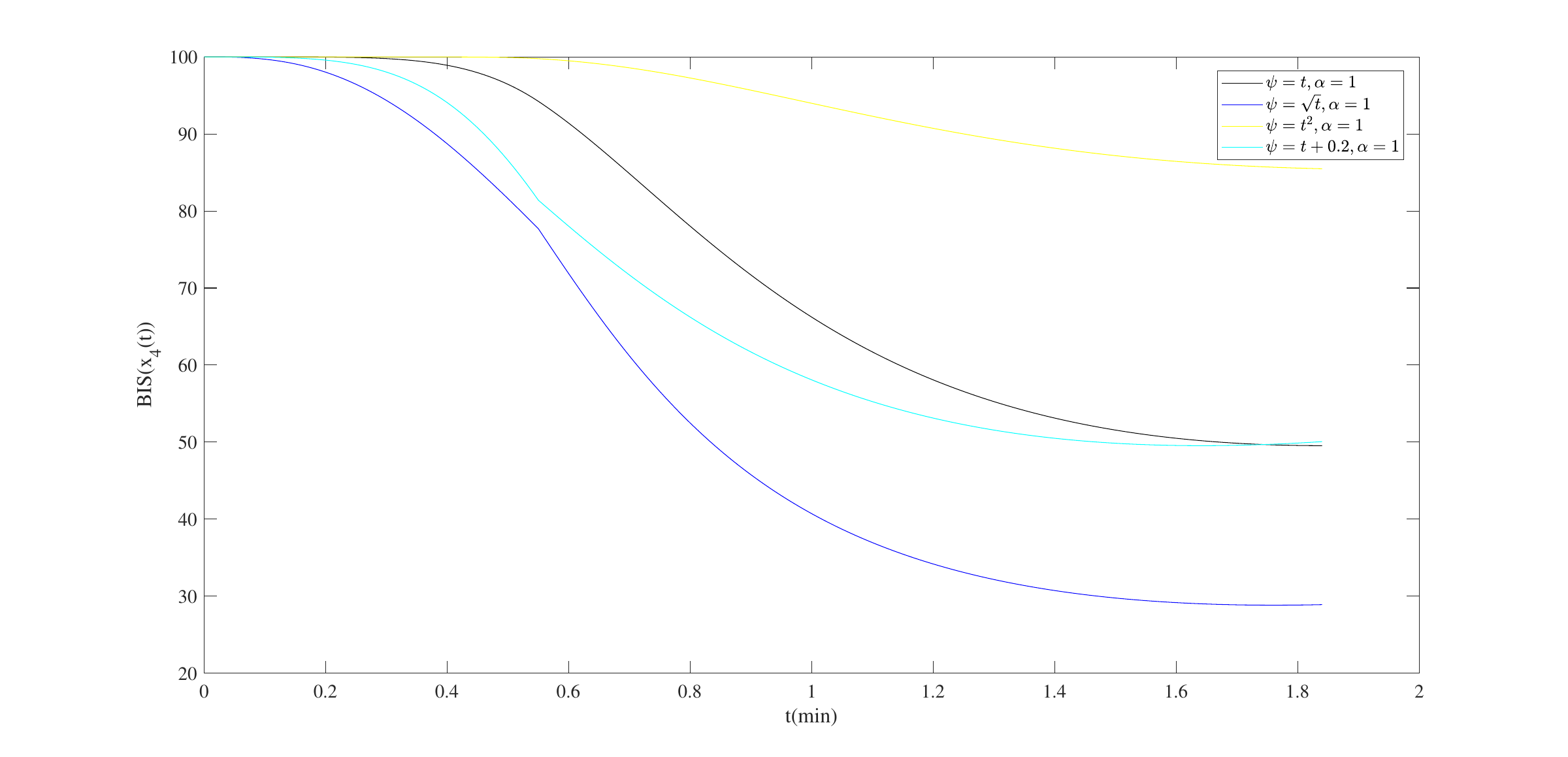}	
\caption{Analysis of controlled BIS with functions $\psi(t)=t$, $\psi(t)=\sqrt{t}$, $\psi(t)=t^2$ 
and $\psi(t)=t+0.2$ for fractional order $\alpha=1$.}
\label{BISI:alpha:1}
\end{center}
\end{figure}
  	 
In Figure~\ref{Solution:alpha:three}, we present the solutions of the fractional PK/PD model 
\eqref{Example1:Time:Optimal} corresponding to the functions $\psi(t)=t$ and $\psi(t)=t+0.2$, 
under the fractional orders $\alpha=1$, $\alpha=0.9$, and $\alpha=0.8$. The curves representing 
the controlled BIS are displayed in Figure~\ref{BISI:alpha:three}. It is worth noting that the 
recorded BIS values for all fractional orders ranged from 50 (resembling the classical model) 
to 60, thereby satisfying the condition of anesthesia. However, it is crucial to acknowledge that 
lower fractional order values, specifically with the function $\psi(t)=t+0.2$, result in a 
reduced risk of awareness during anesthesia.
\begin{figure}[!htb]
\begin{center}
\includegraphics[width=\textwidth,height=\textheight,keepaspectratio]{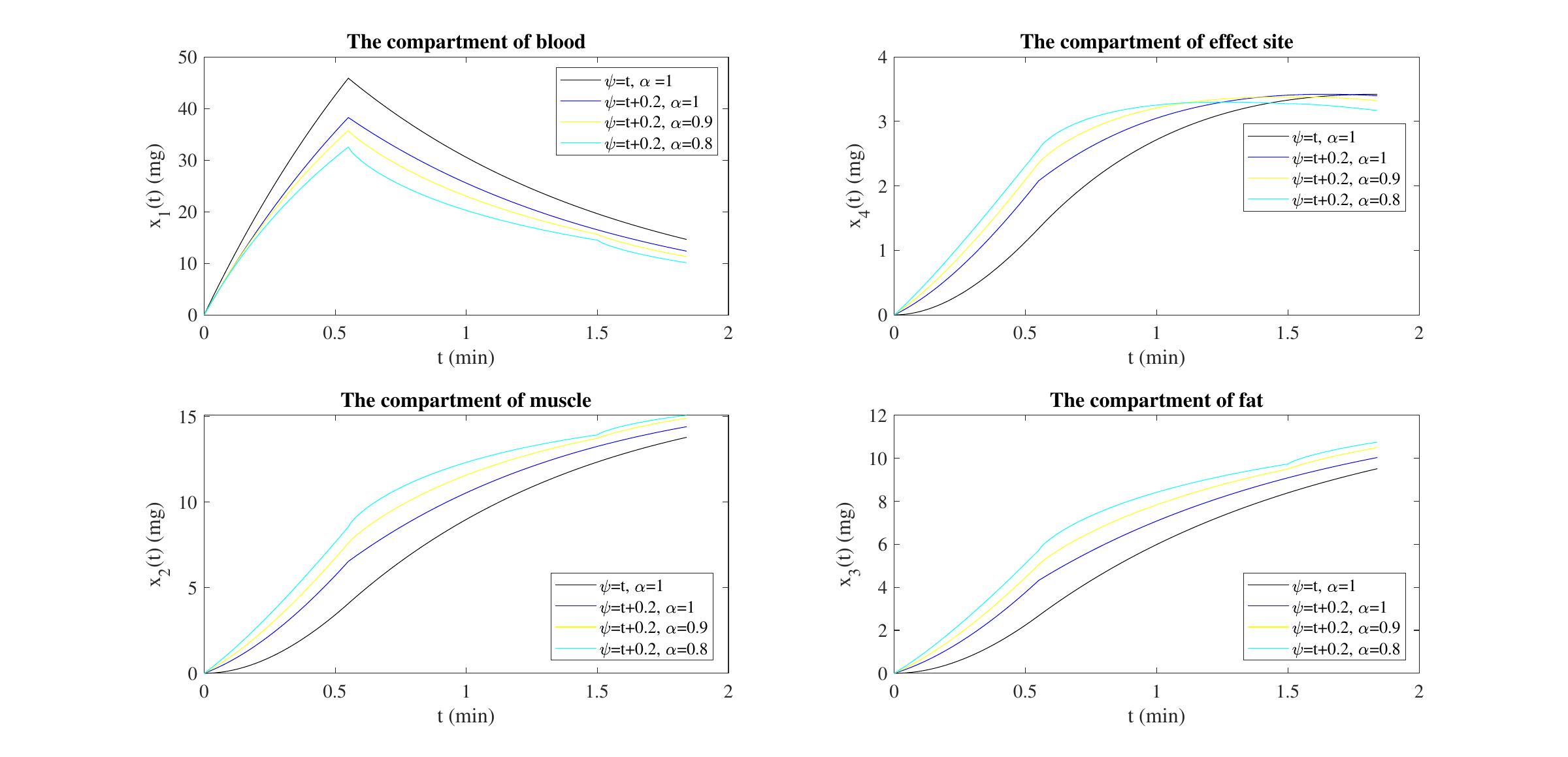}
\caption{Analysis of the fractional PK/PD model \eqref{Example1:Time:Optimal} with functions $\psi(t)=t$ 
and $\psi(t)=t+0.2$ for fractional orders $\alpha=1$, $\alpha=0.9$, and $\alpha=0.8$.} 
\label{Solution:alpha:three}
\end{center}
\end{figure}
\begin{figure}[!htb]
\begin{center}
\includegraphics[width=\textwidth,height=\textheight,keepaspectratio]{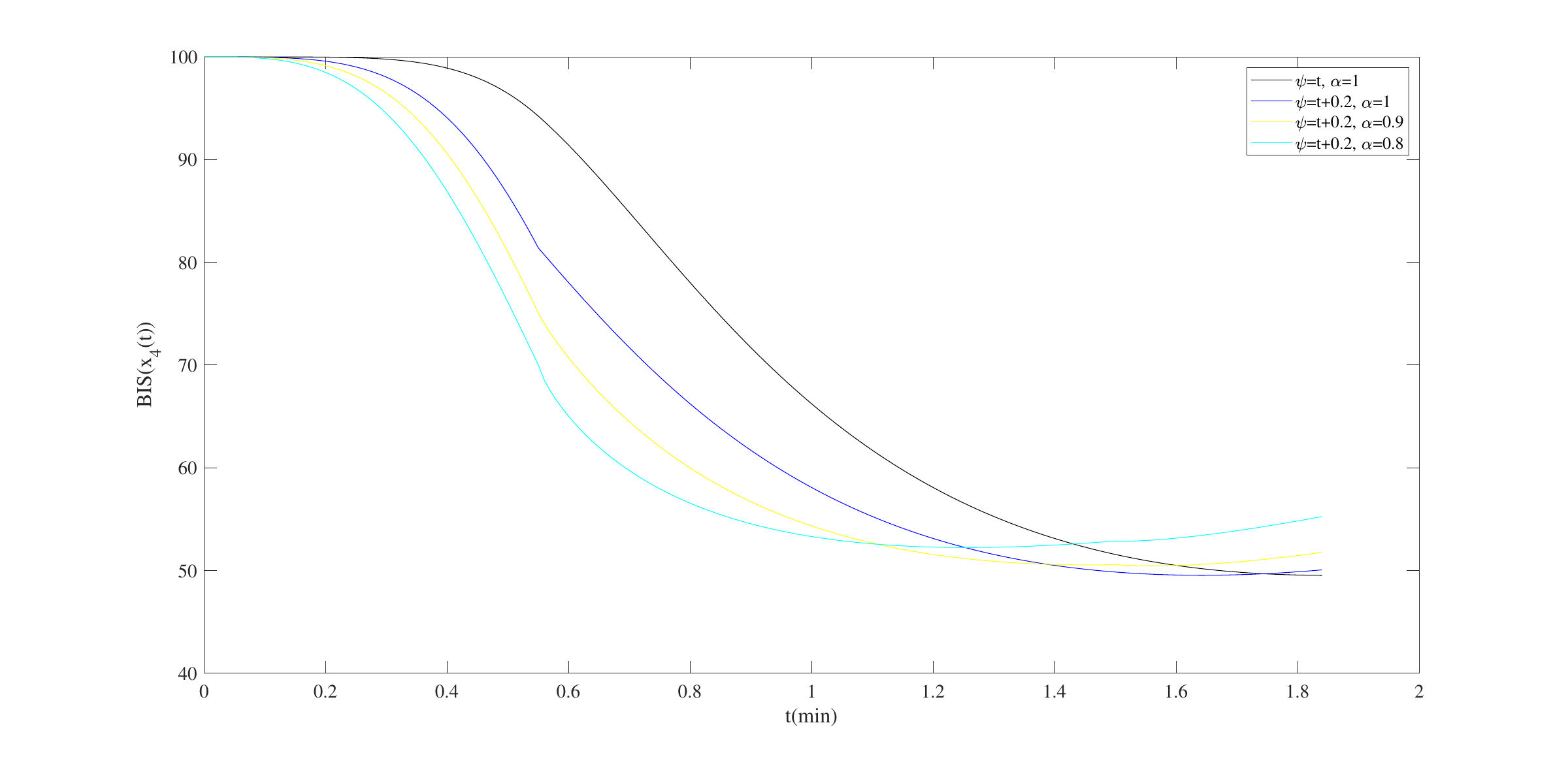}
\caption{Analysis of controlled BIS with functions $\psi(t)=t$ and $\psi(t)=t+0.2$ 
for fractional orders $\alpha=1$, $\alpha=0.9$, and $\alpha=0.8$.}
\label{BISI:alpha:three}
\end{center}
\end{figure}  	 


\section{Discussion}
\label{sec:disc}

The $\psi$-Caputo fractional PK/PD model represents a notable advancement 
in modeling due to its ability to capture intricate drug response dynamics, 
considering the complex relationship between drug concentration and physiological effects. 
In comparison to conventional modeling approaches, the incorporation of the $\psi$-Caputo 
fractional derivative and the fractional order $\alpha$ enable a more comprehensive 
representation of non-local and memory-dependent effects, offering a more nuanced 
understanding of the induction phase of anesthesia. However,
it should be noted that the $\psi$-Caputo fractional PK/PD model 
is more complex compared to conventional modeling approaches \cite{Said,ZST}
because it relies on fractional calculus and an unknown $\psi$ function, 
introducing additional parameters, and requires specialized 
mathematical knowledge. Therefore, its use should be reserved 
for situations where traditional models are insufficient 
for capturing the intricacies of the system being modeled
with conventional models, based on integer-order calculus, 
remaining the standard for most PK/PD modeling tasks due to their 
ease of use and the availability of established methodologies 
and software \cite{Zabi}.

Complex systems often require sophisticated modeling techniques 
that consider intricate interactions and feedback mechanisms. Additionally, the availability 
of high-quality data, incorporating a wide range of variables, significantly enhances 
the reliability of model predictions. Given this, the accuracy and predictability of $\psi$-Caputo 
fractional PK/PD models are contingent upon the complexity of the research domain 
and the quality of available data, presenting some limitations:
(i) additional parameters; and (ii) data requirements.
Indeed, (i) the $\psi$-Caputo model introduces extra parameters associated 
with the fractional orders, which need to be estimated from data. These additional 
parameters can make the model more complex and require more data for accurate estimation.
(ii) Fractional PK/PD models may require more extensive data to accurately estimate 
parameters and capture the complexities of the system. Conventional models, 
in some cases, might work with fewer data points. Moreover, the $\psi$-Caputo 
fractional PK/PD model inherently carries uncertainties related to the selection of the 
function $\psi$ and the fractional order $\alpha$, which 
can impact the model's reliability in predicting anesthesia dosage. Uncertainties 
in this model can arise from various sources,  e.g.
(i) measurement uncertainty (errors in drug concentration measurements 
can propagate into the model, affecting the accuracy of estimated parameters and predictions),
(ii) inter-individual variability (patients can have varying physiological 
characteristics, which can lead to inter-individual variability in drug response, 
with the model not capturing all aspects of this variability, 
potentially resulting in suboptimal dosing for some patients),
(iii) intra-individual variability (even within the same individual, the response 
to anesthesia drugs can vary due to factors like changes in organ function, 
health status, or concurrent medications, and in that case the $\psi$-Caputo 
fractional model may not account for all these factors, leading to uncertainty 
in dosing recommendations for the same patient at different times),
(iv) parameter estimation uncertainty (the model relies on various parameters 
such as clearance, volume of distribution, and rate constants, and estimating 
these parameters from experimental data can introduce uncertainty),
(v) data uncertainty (the quality and quantity of data used to estimate model 
parameters can vary, and this can lead to uncertainties in the model; sparse 
or noisy data can result in less reliable model predictions, 
affecting the accuracy of anesthesia dosage recommendations),
(vi) model structure uncertainty (the $\psi$-Caputo fractional PK/PD model 
itself is a simplification of the complex processes occurring in the body,
not fully capturing all relevant mechanisms, leading to uncertainties in predictions).
Consequences of these uncertainties in the case of anesthesia (dosage) 
can be significant: underdosing, if the model underestimates the required dosage, 
in which case patients may not achieve the desired level of anesthesia, 
leading to inadequate pain control or awareness during surgery;
overdosing, if the model overestimates the required dosage, with patients 
receiving an excessive anesthesia, leading to complications 
such as extended recovery times, respiratory depression, or other side effects.

While the $\psi$-Caputo fractional PK/PD model shows advancements 
in capturing complex drug responses, there might be certain types of drugs 
or treatments for which this modeling approach may not be as suitable. 
Some reasons why fractional PK/PD modeling may not be ideal in certain cases
include: (i) non-linear pharmacokinetics; (ii) complex drug interactions; 
(iii) short-acting anesthetics.

Finally, it should be remarked that patient safety must be always of primary 
importance, especially concerning regulatory aspects like the Bispectral Index (BIS). 
Any modeling approach should promote patient safety by accurately predicting 
drug responses and aiding in the development of safer and more effective 
treatment protocols. Some key aspects on patient safety and the regulation 
of PKPD models, with a focus on BIS monitoring, are:
\begin{itemize}
\item Balancing Technology and Clinical Expertise. While PKPD models 
and BIS monitoring are valuable tools in anesthesia, they should 
not replace the judgment and expertise of skilled anesthesiologists. 
The safe administration of anesthesia requires a combination 
of technology and clinical acumen.	
\item Evidence-Based Practice. The development and use of PKPD models 
and monitoring devices like BIS should be based on rigorous scientific 
evidence and clinical studies. Regulatory bodies should require strong 
evidence of safety and efficacy before approving or endorsing these technologies.	
\item Regulatory Oversight. Regulatory agencies play a crucial role in ensuring 
that medical devices and models are safe and effective. They should establish 
and enforce guidelines for the development and use of PKPD models and monitoring 
devices, including BIS. Regular assessments and updates should be 
conducted to account for evolving scientific knowledge.
\item Continuous Monitoring. Real-time monitoring of patients during surgery 
is essential for patient safety. Monitoring parameters, including BIS values, 
should be continuously observed and interpreted by trained personnel 
to respond promptly to any adverse events or deviations.	
\item Individualized Care. Each patient is unique, and anesthesia management 
should be tailored to their specific needs and responses. PKPD models can provide 
guidance, but individualized care remains central to patient safety.
\end{itemize}

 	 
\section{Conclusion}
\label{sec:conc}

The incorporation of the $\psi$-Caputo fractional derivative 
in Pharmacokinetics and Pharmacodynamics modeling represents 
a significant advancement in the field. Indeed, by utilizing 
fractional-order derivatives, researchers can more accurately 
capture the complex and non-local behavior observed 
in drugs within biological systems.

The choice of the function $\psi$ and the fractional order $\alpha$ holds critical 
importance in modeling the relationship between drug concentrations and pharmacological effects. 
This approach provides a more realistic representation of drug efficacy and dose-response 
relationships, allowing for a deeper understanding of the intricate dynamics 
involved in drug-target interactions.

While our study successfully demonstrates the potential of the $\psi$-Caputo 
fractional PK/PD model in capturing complex drug responses during the induction 
phase of anesthesia, it is crucial to acknowledge certain limitations. The current 
model's applicability may be constrained by the specific parameters chosen 
and the simplifications employed to facilitate numerical analysis. Furthermore, 
the availability of comprehensive clinical data for model validation remains 
a challenge, which could affect the generalizability of the findings. 
Future research efforts should focus on addressing these limitations 
to further enhance the applicability and robustness of the proposed model. 
Moreover, further research is necessary to explore the impact of the chosen function 
$\psi$ and the fractional order $\alpha$ on time-delayed responses. This area remains 
open for investigation, and future studies can delve into understanding how different 
choices of $\psi$ and $\alpha$ influence the temporal aspects of drug responses.

In summary, the incorporation of $\psi$-Caputo fractional derivatives in Pharmacokinetics 
and Pharmacodynamics modeling offers valuable insights and advancements. By refining the 
choice of function $\psi$ and fractional order $\alpha$, researchers can enhance the accuracy 
and realism of drug modeling, paving the way for a more comprehensive understanding of drug 
behavior in biological systems. 


\section*{Acknowledgments}

This research was developed within the project 
``Mathematical Modelling of Multi-scale Control Systems: 
applications to human diseases (CoSysM3)'', 2022.03091.PTDC, 
financially supported by national funds (OE), through FCT/MCTES.
The authors are also supported by FCT and CIDMA via projects 
UIDB/04106/2020 and UIDP/04106/2020. 
The authors are grateful to three reviewers
for several comments, suggestions and questions, 
that helped them to improve the initial submitted manuscript.



\end{document}